\documentclass[a4paper,10pt]{article}

\usepackage[utf8]{inputenc}
\usepackage[T1]{fontenc}
\usepackage{mathtools,amsthm,amssymb,amsmath}
\usepackage{lmodern}
\usepackage{geometry}
\usepackage{graphicx}
\usepackage{xcolor}
\usepackage{enumerate,enumitem,listings}
\usepackage{complexity}
\usepackage{blkarray}
\usepackage{hyperref}

\usepackage[font=small,labelfont=bf]{caption}
\usepackage[font=small,labelfont=normalfont,labelformat=simple]{subcaption}
\graphicspath{{figs/}}

\newtheorem{theorem}{Theorem}

\newtheorem{observation}{Observation}

\newtheorem{conjecture}{Conjecture}

\newcounter{cclaim}
\theoremstyle{definition}
\newtheorem{claim}[cclaim]{Claim}

\newenvironment{claimproof}[1][\proofname]
{%
  \proof[#1]%
}
{%
  \endproof
}

\author
{
Stéphane Bessy \thanks{LIRMM, CNRS, Université de Montpellier, France, \texttt{Stephane.Bessy@lirmm.fr.} Supported by the French ANR project ANR-24-CE48-4377 (GODASse)}
\and 
Matthijs Muis \thanks{Department of Mathematics, ETH Z\"{u}rich, Switzerland,  \texttt{mamuis@ethz.ch}.}
\and
Jean-Sébastien Sereni \thanks{Centre national de la recherche scientifique, Institut de recherche mathématique avancée,
Strasbourg, France, \texttt{jean-sebastien.sereni@cnrs.fr}}
\and
Raphael Steiner \thanks{Department of Mathematics, ETH Z\"{u}rich, Switzerland,  \texttt{raphaelmario.steiner@math.ethz.ch}.
Supported by the SNSF Ambizione Grant No. 216071}
\and
Sebastian Wiederrecht \thanks{School of Computing, KAIST, Daejeon, South Korea \texttt{wiederrecht@kaist.ac.kr}.
Supported by the Institute for Basic Science (IBS-R029-C1)}
}
\date{\today}

\title{A relaxation of the Bermond-Thomassen conjecture}

\begin{document}
\maketitle

\begin{abstract}
The well-known Bermond-Thomassen conjecture states that every digraph of minimum out-degree at least~$2k-1$
contains~$k$ vertex-disjoint directed cycles. Despite being posed in 1981, this conjecture remains unresolved
for all~$k \ge 4$.

We prove a relaxation of this conjecture: every digraph~$D$ of minimum out-degree at least~$2k-1$ contains~$k$
vertex-disjoint cycles, each of which either is directed or can be made directed by reversing one of its arcs.
This bound is sharp and answers a question raised by Cames van Batenburg during the online workshop
``Entropy Compression and Related Methods'' in~$2021$.
\end{abstract}

\section{Introduction}

The study of degree conditions enforcing certain substructures has a long history in graph theory. For
example, the famous theorems of Tur\'{a}n~\cite{turan} and Dirac~\cite{MR47308} determine the best-possible
average and minimum degree conditions guaranteeing the presence of a complete subgraph or of a Hamiltonian
cycle, respectively, which both turn out to be linear in the number of the vertices of the host graphs. When
looking for sparser structures, such as subdivisions of fixed graphs~\cite{MR1657911,MR1476462,MR1395694}, a
constant average degree is already sufficient. A special case of such sparse structures which has been studied
in depth for undirected graphs are packings of a prescribed number of vertex-disjoint cycles (often with
additional properties), see the survey by Chiba and Yamashita~\cite{chibasurvey} for a somewhat recent account
of the large body of work on this topic. In particular, it has been known for a long time (and it is not
difficult to prove) that in undirected graphs, a minimum degree of~$3k-1$ guarantees the existence of ~$k$
vertex-disjoint cycles, and this bound is best-possible as can be verified by the example of the complete
graph~$K_{3k-1}$, which has minimum degree~$3k-2$ but no~$k$ vertex-disjoint cycles. The classical
Corrádi-Hajnal theorem~\cite{corradihajnal} from~$1963$ shows that this minimum degree condition can however be
improved from~$3k-1$ to~$2k$ when the host graph is assumed to have at least~$3k$ vertices.

In contrast, for packings of vertex-disjoint \emph{directed} cycles in \emph{directed} host graphs,
comparatively little is known. Since high average or minimum degree of the underlying graph of a digraph does
not by itself force even the existence of a single directed cycle (consider a transitive orientation of a
complete graph), here it is more natural to impose conditions on the minimum \emph{out-}degree~$\delta^+(D)$
of the host digraph~$D$. Then the natural analogous problem arises: what is the best-possible minimum
out-degree condition enforcing the existence of $k$ vertex-disjoint directed cycles? In a famous conjecture
from~$1981$, Bermond and Thomassen~\cite{bermondthomassen} conjectured the precise value to be~$2k-1$. 
\begin{conjecture}[Bermond-Thomassen conjecture~\cite{bermondthomassen}]
Let~$k$ be a positive integer. Every digraph~$D$ with $\delta^+(D)\ge 2k-1$ contains~$k$ vertex-disjoint directed cycles. 
\end{conjecture}

This conjecture, if true, would indeed be best-possible, since the complete digraph on~$2k-1$ vertices (i.e.,
with all possible ordered pairs of distinct vertices as arcs) has minimum out-degree~$2k-2$ but is too small
to host~$k$ vertex-disjoint directed cycles (each of which must use at least two vertices). Despite
its innocent appearance and the fact that it has received significant attention, the Bermond-Thomassen conjecture has turned out to be
tremendously difficult over the decades, with even partial progress proving to be challenging. In fact, it is
not obvious at all that there would even exist \emph{any} function~$b:\mathbb{N}\rightarrow \mathbb{N}$ such
that every digraph of minimum out-degree at least~$b(k)$ contains~$k$ vertex-disjoint directed cycles. That
this is indeed true was first proved by Thomassen~\cite{ThomassenDisjointCycles} using a beautiful contraction
trick, proving an upper bound on~$b(k)$ growing roughly like~$(k+1)!$. So far, the precise form of the
conjecture has only been proved for~$k=2$ by Thomassen in the same~$1983$ article~\cite{ThomassenDisjointCycles},
and more recently for~$k=3$ by Lichiardopol, P\'{o}r and Sereni~\cite{k3}. It remains wide open for
each~$k\ge 4$, and only for restricted classes, such as tournaments~\cite{bangjensen}, has the conjecture
been proved in full generality. 

Interestingly, Thomassen~\cite{ThomassenDisjointCycles} also proved that for every fixed value of~$k$,
deciding the Bermond-Thomassen conjecture boils down to verifying it for all digraphs on at most a bounded
number of vertices (the bound growing with~$k$). Nevertheless, a brute-force approach seems to be practically
infeasible already for~$k=3$, even with access to modern computing power.

As for asymptotic upper bounds on~$b(k)$, significant progress has been made. Thomassen's fast-growing
factorial bound was later significantly improved to a linear bound of~$b(k)\le 64 k$ by Alon~\cite{alon}, who
introduced several clever probabilistic ideas on top of Thomassen's contraction trick. Alon's proof approach
was then quantitatively refined and optimized by Buci\'{c}~\cite{bucic}, who proved the current
state-of-the-art upper bound~$b(k)\le 18k$. It seems unlikely that further optimization of the methods used by
Alon and Buci\'{c}, without substantial new ideas, could lead to a proof of the precise version of the
Bermond-Thomassen conjecture.

In this note, we shall prove a novel relaxation of the Bermond-Thomassen conjecture, which answers
a question raised by Cames van Batenburg at the~$2021$ online workshop
\href{https://sparse-graphs.mimuw.edu.pl/doku.php?id=sessions:2021sessions:2021session2}{``Entropy Compression and Related Methods''}.
Instead of relaxing the minimum out-degree condition as in the aforementioned
results, we only require the same bound~$2k-1$ as in the Bermond-Thomassen conjecture, but instead slightly
relax the requirement that the found cycles be directed: call a cycle in a digraph~$D$ \emph{near-directed} if
either it is directed or it can be made directed by reversing one of its arcs. See Figure~\ref{fig:neardirected}
for an illustration. We then prove the following result.
\begin{theorem}\label{thm:main1}
Let~$k$ be a positive integer. Every digraph~$D$ with~$\delta^+(D) \ge 2k-1$ contains~$k$ vertex-disjoint near-directed cycles.
\end{theorem}
\begin{figure}
    \centering
    \includegraphics[width=0.7\linewidth]{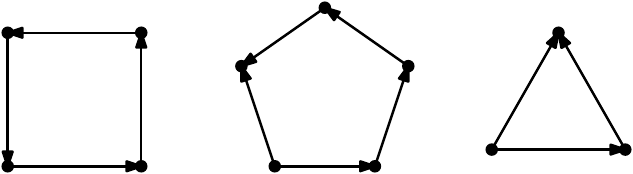}
    \caption{Three vertex-disjoint near-directed cycles.}
    \label{fig:neardirected}
\end{figure}

\noindent Note that while being qualitatively weaker than the Bermond-Thomassen conjecture, the statement of
Theorem~\ref{thm:main1} is still \emph{quantitatively} optimal: the complete digraph on~$2k-1$ vertices has
minimum out-degree~$2k-2$ and no collection of~$k$ vertex-disjoint near-directed cycles, since every such cycle uses at least two vertices.

\paragraph*{Notation and Terminology.} For a digraph~$D$, we denote by~$V(D)$ its vertex set and
by~$A(D)\subseteq \{(u,v)\in V(D)\times V(D)\mid u\neq v\}$ its arc set. For an arc~$(u,v)$, we think of~$u$ as
its starting point (tail) and~$v$ as its endpoint (head). We also set~$v(D)=|V(D)|$ and~$a(D)=|A(D)|$.
For a vertex~$v\in V(D)$, we let~$N_D^+(v)$ and~$N_D^-(v)$ be
its out- and in-neighborhood, respectively, and let $d_D^+(v)$, $d^-_D(v)$ denote the out- and in-degree, respectively; we may drop the subscript~$D$ if clear from context.
If~$u$ is an out-neighbor of~$v$ (in~$D$), we sometimes write that \emph{$v$ sees~$u$} (in~$D$).
For a subset of vertices~$X$ of~$D$, we let~$D[X]$ be the subdigraph of~$D$ induced by the vertices in~$X$,
that is the digraph~$(X,A(D)\cap(X\times X))$. A subdigraph~$D'$ is \emph{induced} (in~$D$) if~$D' = D[V(D')]$.
Similarly, if~$Y$ is a set of vertices or arcs of~$D$, then we denote by~$D-Y$ the digraph obtained from~$D$ by
deleting the elements in~$Y$ (along with the incident arcs if the element is a vertex).
We treat directed paths and cycles in the usual sense and as subdigraphs of the
relevant host digraphs. Directed paths are allowed to be of length~$0$ (a single vertex), and directed cycles
are allowed to be of length~$2$, in which case they consist of two vertices connected by antiparallel arcs.
However, near-directed cycles that are not directed cycles are necessarily of length at least~$3$. If~$x$ and~$y$ are two vertices of a directed path~$P$,
then the \emph{directed path between~$x$ and~$y$ along~$P$} is the directed path with extremal vertices~$x$ and~$y$ contained in~$P$.
If~$Q$ is a directed path or a cycle in a digraph, we let~$|Q|$ be its length (number of arcs).
We let~$\delta^+(D)$ be the minimum out-degree of~$D$, i.e., the number of out-neighbors, minimized over all
vertices of the digraph. We end this paragraph with a straightforward observation, which will be useful to
invoke later.

\begin{observation}\label{obs}
Let~$P$ be a directed path in a digraph~$D$.
\begin{enumerate}
  \item\label{obs-indpath} If~$P$ is not induced, then the subdigraph of~$D$ induced by~$V(P)$ contains a near-directed cycle.
  \item\label{obs-pathplusvertex} If~$v\in V(D)\setminus V(P)$ such that~$|N_D^+(v)\cap V(P)|\ge 2$ or~$|N_D^-(v)\cap V(P)|\ge 2$,
    then the subdigraph induced by~$V(P)\cup\{v\}$ contains a near-directed cycle.
\end{enumerate}
\end{observation}
\begin{claimproof}
  Let~$v_1,\dotsc,v_\ell$ be the vertices of~$P$ ordered such that~$(v_i,v_{i+1})\in A(D)$ for~$i\in\{1,\dotsc,\ell-1\}$.
  If~$P$ is not induced in~$D$, then there exist~$i,j\in\{1,\dotsc,\ell\}$ such that~$a\coloneqq (v_i,v_j)\in A(D)$ and~$j\ge i+2$ or $j\le i-1$. 
  Therefore, the arc~$a$ along with the directed path between~$v_i$ and~$v_j$ along~$P$
  form a near-directed cycle in~$D$.  This proves~\ref{obs-indpath}.

  As for~\ref{obs-pathplusvertex}, up to reversing all arcs we can assume
  that~$|N_D^+(v)\cap V(P)|\ge 2$. Then there exist distinct indices~$i,j\in\{1,\dotsc,\ell\}$
  such that~$\{v_i,v_j\}\subseteq N_D^+(v)$. Therefore, the directed path between~$v_i$ and~$v_j$ along~$P$
  together with the arcs~$(v,v_i)$ and~$(v,v_j)$ form a near-directed cycle in~$D$.
\end{claimproof}

\section{Proof of Theorem\texorpdfstring{~\ref{thm:main1}}{ 1}}\label{sec:neardirected}
We now present the self-contained proof of Theorem~\ref{thm:main1}. The proof method is inspired by the proof of a somewhat similar, but not directly related, cycle packing result for \emph{undirected graphs} due to Enomoto~\cite{undirected}, but needs to address several new challenges coming from working in the realm of directed graphs.
\begin{proof}[Proof of Theorem~\ref{thm:main1}]
Suppose towards a contradiction that the statement is wrong, i.e., there exists a (smallest) integer~$k \ge 1$
for which not all digraphs of minimum out-degree at least~$2k-1$ contain~$k$ vertex-disjoint near-directed
cycles. Clearly,~$k\ge 2$.

Let~$D$ be a digraph with~$\delta^+(D) \ge 2k-1$ and no collection of~$k$ vertex-disjoint near-directed
cycles, chosen among all such digraphs so as to minimize the number of vertices~$v(D)$, and among all such
digraphs with the same number of vertices, so as to maximize the number of arcs~$a(D)$.

Our goal is to derive a contradiction to our initial assumption. Let us start with some simple
observations.

\begin{claim}\label{c1}
  If~$(x,y) \in A(D)$ for some pair of vertices~$x,y$, then~$(y,x) \notin A(D)$.
\end{claim}
\begin{claimproof}
Suppose towards a contradiction that~$(x,y), (y,x) \in A(D)$. Let~$C_1$ be the directed cycle in~$D$ induced
by~$x$ and~$y$, and let~$D'\coloneqq D-\{x,y\}$. Then~$\delta^+(D') \ge \delta^+(D)-2 \ge
(2k-1)-2=2(k-1)-1$, and hence, by the assumed minimality of~$k$, we know that~$D'$ contains a
collection~$C_2,\ldots,C_k$ of~$(k-1)$ vertex-disjoint near-directed cycles. This however means
that~$C_1,C_2,\ldots,C_k$ form a collection of~$k$ vertex-disjoint near-directed cycles in~$D$, contradicting
our initial assumptions on~$D$.
\end{claimproof}
\begin{claim}\label{c2}
  The digraph~$D$ is not a tournament.
\end{claim}
\begin{claimproof}
  As reported earlier, a stronger statement than that of Theorem~\ref{thm:main1} is true for tournaments~\cite{bangjensen}.
  However we provide a short proof in our case for completeness. Suppose, on the contrary, that~$D$ is a tournament.
  Since~$\delta^+(D) \ge 2k-1$, it follows that the tournament~$D$ contains at least~$2(2k-1)+1$ vertices,
  which is at least~$3k$. Consequently, the underlying undirected graph~$G$ of~$D$ is a complete graph of order at least $3k$, in which one readily obtains~$k$ vertex-disjoint triangles. Each triangle in~$G$ yields either
  a directed triangle or a transitive triangle in~$D$, which are both near-directed cycles.
  Therefore~$D$ contains~$k$ vertex-disjoint near-directed cycles, the desired contradiction.
\end{claimproof}

\begin{claim}\label{c3}
  There exist~$k-1$ vertex-disjoint near-directed cycles~$C_1,\dotsc,C_{k-1}$ in~$D$ such that $| \bigcup_{j=1}^{k-1}{V(C_j)}| \le v(D)-3$.
\end{claim}
\begin{claimproof}
By Claim~\ref{c1}, $D$ contains no antiparallel pair of arcs, and by Claim~\ref{c2}, there are distinct vertices~$x,y \in
V(D)$ such that $(x,y), (y,x) \notin A(D)$. Let~$D'$ be defined as the digraph obtained from~$D$ by adding
the arc~$a\coloneqq (x,y)$. Then~$\delta^+(D') \ge \delta^+(D) \ge 2k-1$, $v(D')=v(D)$ and~$a(D')>a(D)$,
so our choice of~$D$ implies that~$D'$ must contain a collection of~$k$ vertex-disjoint near-directed
cycles. Clearly, at most one cycle in the collection can contain the arc~$a$, so we may denote the
collection as~$C_1,\dotsc,C_{k-1},C_k$, where~$a \notin \bigcup_{j=1}^{k-1}A(C_j)$. Then~$C_1,\dotsc,C_{k-1}$
is a collection of $k-1$ vertex-disjoint near-directed cycles in $D$. Furthermore,~$|C_k|\ge 3$
since~$(y,x) \notin A(D')$, and hence~$|\bigcup_{j=1}^{k-1}{V(C_j)}| \le v(D')-|V(C_k)| \le v(D)-3$.
\end{claimproof}

\begin{claim}\label{c4}
Let~$C_1,\dotsc,C_{k-1}$ be vertex-disjoint near-directed cycles in~$D$ that
together cover the fewest vertices, i.e. such that~$|\bigcup_{j=1}^{k-1}{V(C_j)}|$ is
minimized. If~$i \in \{1,\dotsc,k-1\}$ and~$x \notin \bigcup_{j=1}^{k-1}{V(C_j)}$ such
that~$|N^+(x) \cap V(C_i)|\ge 3$, then~$|V(C_i)|=3$.
\end{claim}

\begin{claimproof}
By Claim~\ref{c1}, we have~$|V(C_i)| \ge 3$.
Suppose towards a contradiction that~$|V(C_i)| \ge 4$. Since~$C_i$ is a near-directed cycle, there exists a
directed path~$Q \subseteq C_i$ such that~$V(Q)=V(C_i)$. Let~$x_1, x_2 \in N^+(x) \cap V(Q)$ be chosen so as
to minimize the distance between~$x_1$ and~$x_2$ along the path~$Q$. Since~$|N^+(x) \cap V(Q)|=|N^+(x) \cap V(C_i)|
\ge 3$, the distance~$d$ between~$x_1$ and~$x_2$ along~$Q$ satisfies $d \le
\frac{1}{2}|Q|=\frac{1}{2}(|V(C_i)|-1)$, and hence the subdigraph~$C$ of~$D$ formed by the union of the path
between~$x_1$ and~$x_2$ along~$Q$ and the arcs~$(x,x_1),(x,x_2)$ satisfies $|C|=d+2 \le
\frac{1}{2}(|V(C_i)|+3)<|V(C_i)|$, where for the
last inequality we used that $|V(C_i)| \ge 4$. Note that~$C$ is a near-directed cycle in~$D$.
Since~$x \notin \bigcup_{j=1}^{k-1}{V(C_j)}$, this means that~$C_1,\dotsc,C_{i-1},C,C_{i+1},\dotsc,C_{k-1}$ is
a collection of~$k-1$ vertex-disjoint near-directed cycles in~$D$ which covers fewer vertices
than~$C_1,\ldots,C_{k-1}$, a contradiction.
\end{claimproof}

For the rest of the proof, let us fix a collection~$\mathcal{C}=\{C_1,\dotsc,C_{k-1}\}$ of vertex-disjoint near-directed
cycles in~$D$ that
\begin{enumerate}
  \item\label{colloc1} covers the fewest vertices; and
\item\label{colloc2} subject to~\ref{colloc1}, maximizes
the length of a longest directed path contained in the digraph~$D-\bigcup_{j=1}^{k-1}{V(C_j)}$
(which contains no directed cycle by our assumption on~$D$).
\end{enumerate}
Claim~\ref{c3} implies that~$| \bigcup_{j=1}^{k-1}{V(C_j)}| \le v(D)-3$.
This extremal choice of the collection~$C_1,\dotsc,C_{k-1}$ guarantees several further properties, stated in
the following claims. 
We set~$X\coloneqq V(D)\setminus \bigcup_{j=1}^{k-1}{V(C_j)}$.

\begin{claim}\label{c5}
  The induced subdigraph~$D[X]$ is a directed path.
\end{claim}
\begin{claimproof}
  Let~$P$ be a directed path in~$D[X]$ of maximal length. Since~$D[X]$ contains no near-directed cycle,
  we note that~$P$ must be an \emph{induced} subgraph of~$D$, that is,~$D[V(P)]=P$, by
  Observation~\ref{obs}-\ref{obs-indpath}.

We now show that~$D[X]=P$, and hence that~$D[X]$ is a directed path, by proving that~$V(P)=X$.
Suppose, on the contrary, that~$V(P) \neq X$. Let~$Y\coloneqq X \setminus V(P)$, so~$Y \neq \varnothing$.
As noted earlier,~$D[X]$, and hence~$D[Y]$, must be an acyclic digraph (i.e. cannot contain a directed cycle).
It follows that~$D[Y]$ must contain a sink~$y \in Y$. This means that $N^+(y) \subseteq
\bigcup_{j=1}^{k-1}{V(C_j)} \cup V(P)$. We furthermore deduce that $|N^+(y) \cap V(P)| \le 1$
from Observation~\ref{obs}-\ref{obs-pathplusvertex}.

As a consequence,~$\sum_{i=1}^{k-1}{|N^+(y) \cap V(C_i)|} \ge d^+(y)-1 \ge 2k-2$. In addition,
letting~$x \in V(P)$ be the sink of the directed path~$P$,
we must have $N^+(x) \subseteq \bigcup_{j=1}^{k-1}{V(C_j)}$
because~$P$ was chosen as a longest
directed path in~$D[X]$, and~$P$ is induced in~$D$.
Therefore, $\sum_{i=1}^{k-1}{|N^+(x) \cap V(C_i)|}=d^+(x) \ge 2k-1$. Altogether,
we obtain
\[\sum_{i=1}^{k-1}{\left(|N^+(x) \cap V(C_i)|+|N^+(y) \cap V(C_i)|\right)} \ge 4k-3>4(k-1).\]
By averaging, there exists~$i \in \{1,\ldots,k-1\}$ such that $|N^+(x) \cap V(C_i)|+|N^+(y) \cap V(C_i)| \ge
5$. This implies that $|N^+(x) \cap V(C_i)| \ge 3$ or $|N^+(y) \cap V(C_i)| \ge 3$, and hence by
Claim~\ref{c4} we have $|V(C_i)|=3$. Therefore each of~$x$ and~$y$ has at least~$2$
out-neighbors in~$V(C_i)$, and one of them (at least) has~$3$.
It follows that we can write~$V(C_i)=\{u,y_1,y_2\}$ such that~$u\in N^+(x)$
and~$\{y_1,y_2\}\subseteq N^+(y)$.
Consequently,~$\{y,y_1,y_2\}$ induce a near-directed cycle~$C$ in~$D$, specifically a transitive triangle,
and~$\mathcal{C}'=\{C_1,\dotsc,C_{i-1},C,C_{i+1},\dotsc,C_{k-1}\}$
is a collection of~$k-1$ vertex-disjoint near-directed
cycles in~$D$ which covers the same amount of vertices as~$\mathcal{C}$ (since~$|V(C_i)|=|V(C)|=3$).
Moreover, adding the arc~$(x,u)$ to the path~$P$ produces a directed path of length~$|P|+1$
contained in~$V(D)\setminus\bigcup_{C'\in\mathcal{C}'}V(C')$,
which contradicts our initial choice of~$\mathcal{C}$.
\end{claimproof}

In the remainder of the proof, let~$P$ be the directed path such that~$D[X]=P$. Let us further
denote the sequence of vertices along~$P$ by~$x_1,x_2,\dotsc,x_t$, such that~$x_1, x_t$ are the extremal vertices
of~$P$ (here, importantly, we do not fix the direction of~$P$, i.e., we only require that~$P$ is directed
(possibly from~$x_1$ towards~$x_t$ or from~$x_t$ towards~$x_1$), this relaxed assumption on the
sequence~$x_1,\dotsc,x_t$ will yield us more freedom at later stages of the proof in order to reduce away some
symmetric cases.

Claim~\ref{c3} and our choice of the collection~$C_1,\ldots,C_{k-1}$ imply that
$t=|V(P)|=|V(D)\setminus\bigcup_{j=1}^{k-1}{V(C_j)}| \ge 3$. 

\begin{claim}\label{c6} 
There exists~$i \in \{1,\dotsc,k-1\}$ such that $|N^+(x_1) \cap V(C_i)|+|N^+(x_t) \cap V(C_i)| \ge 5$ and~$|V(C_i)|=3$. 
\end{claim}
\begin{claimproof}
Consider the vertices~$x_1$ and~$x_t$. Since~$D[X]=P$ and since~$x_1$ and~$x_t$ are the extremal vertices of~$P$, only
one of them has an out-neighbour in~$V(P)$, and hence
\begin{align*}
\sum_{i=1}^{k-1}{\left(|N^+(x_1) \cap V(C_i)|+|N^+(x_t) \cap V(C_i)|\right)} &=|N^+(x_1)\setminus X|+|N^+(x_t)\setminus X|\\
&\ge 2(2k-1)-1 =4k-3\\
&>4(k-1).
\end{align*}
By averaging, this implies that there exists~$i \in \{1,\dotsc,k-1\}$ such that
$|N^+(x_1) \cap V(C_i)|+|N^+(x_t) \cap V(C_i)| \ge 5$. This means that $\max\{|N^+(x_1) \cap
V(C_i)|,|N^+(x_t) \cap V(C_i)|\} \ge 3$, and thus~$|V(C_i)|=3$ by Claim~\ref{c4}.
\end{claimproof}

Up to relabelling~$x_1,\dotsc,x_t$ in reverse order, we may assume that~$|N^+(x_1) \cap V(C_i)|=3$
and~$|N^+(x_t) \cap V(C_i)| \ge 2$. In addition, up to relabelling the cycles~$C_1,\dotsc,C_{k-1}$, we may
assume that~$i=k-1$. For convenience, we use the
notation~$W\coloneqq V(C_{k-1}) \cup V(P)$. We now write~$V(C_{k-1}) = \{u,u',x_0\}$
such that~$\{u,u'\}\subseteq N_D^+(x_t)$.
\begin{claim}\label{c7}
  The digraph~$D[W]$ contains no two vertex-disjoint near-directed cycles.
\end{claim}
\begin{claimproof}
If~$D[W]$ contained two vertex-disjoint near-directed cycles, then these two cycles together
with~$C_1,\dotsc,C_{k-2}$ would yield~$k$ vertex-disjoint near-directed cycles in~$D$, a contradiction to
our initial assumptions on~$D$.
\end{claimproof}
\begin{claim}\label{c8}The following hold.
\begin{enumerate}[label=(\roman*)]
  \item\label{c8i} $|N^+(x_0) \cap W| \le 3$;
  \item\label{c8ii} $|N^+(x_1) \cap W| \le 4$ and $|N^+(x_t) \cap W| \le 4$;
  \item\label{c8iii} $ |N^+(x_2) \cap W| \le 1$.
\end{enumerate}
\end{claim}

\begin{claimproof}\mbox{}

\noindent
\ref{c8i}
Suppose, on the contrary, that~$|N^+(x_0) \cap W| \ge 4$.
Since~$C_{k-1}$ is a triangle containing~$x_0$, it follows that~$|N^+(x_0) \cap V(P)| \ge 2$.
Since~$(x_1,x_0)\in A(D)$, Claim~\ref{c1} implies that~$(x_0,x_1)\notin A(D)$,
and hence Observation~\ref{obs}-\ref{obs-pathplusvertex}
yields that the subdigraph of~$D[W]$ induced by~$\{x_0\}\cup(V(P)\setminus\{x_1\})$
contains a near-directed cycle~$C$. Noticing that~$D[\{x_1,u,u'\}]$ is a near-directed cycle in~$D[W]$
that is vertex-disjoint from~$C$ yields a contradiction with Claim~\ref{c7}.

\noindent
\ref{c8ii}
Let~$i \in \{1,t\}$.
We have $|N^+(x_i) \cap V(C_{k-1})| \le |V(C_{k-1})|=3$ and, since~$P$ is
an induced directed path in~$D$ with extremal vertices~$x_1$ and~$x_t$,
we also have~$|N^+(x_i) \cap V(P)| \le 1$, hence the conclusion. 

\noindent
\ref{c8iii}
Since~$P$ is an induced subgraph of~$D$, the vertex~$x_2$ has at most
one out-neighbor in~$V(P)$, and hence it suffices to show that~$N^+(x_2) \cap V(C_{k-1})=\varnothing$.
Towards a contradiction, suppose that there exists~$y \in N^+(x_2) \cap V(C_{k-1})$.
Because~$D[\{x_t,u,u'\}]$ is a near-directed cycle~$C$, Claim~\ref{c7} implies that~$y\neq x_0$,
for otherwise~$D[\{x_0,x_1,x_2\}]$ would be a near-directed in~$D[W]$ disjoint from~$C$.
Consequently, $y\in\{u,u'\}$
Observation~\ref{obs}-\ref{obs-pathplusvertex} then implies that~$D[\{y,x_2,\dotsc,x_t\}]$
contains a near-directed cycle. Since~$D[\{x_1,x_0,y'\}]$, where~$y'$ is the vertex in~$\{u,u'\}\setminus\{y\}$,
is also a near-directed cycle, this contradicts Claim~\ref{c7}.
We conclude that~$|N^+(x_2) \cap W| \le 1$.
\end{claimproof}

At this point we observe that~$k\ge 3$ must hold; indeed, if not, then~$W = V(D)$, which together with
Claim~\ref{c8}~\ref{c8iii} would imply~$d^+(x_2)\le 1$, in contradiction with our initial observation that~$k$
be at least~2.

\begin{claim}\label{c9}
  There exists~$i \in \{1,\ldots,k-2\}$ such that
\[2|N^+(x_0) \cap V(C_i)|+|N^+(x_1) \cap V(C_i)|+2|N^+(x_2) \cap V(C_i)|+|N^+(x_t) \cap V(C_i)| \ge 13.\]
\end{claim}
\begin{claimproof}
Claim~\ref{c8} implies that
\[2|N^+(x_0) \cap W|+|N^+(x_1) \cap W|+2|N^+(x_2) \cap W|+|N^+(x_t) \cap W| \le 2\cdot3+4+2\cdot 1+4=16.\]
It follows that
\begin{align*}
&\sum_{i=1}^{k-2}{\left(2|N^+(x_0) \cap V(C_i)|+|N^+(x_1) \cap V(C_i)|+2|N^+(x_2) \cap V(C_i)|+|N^+(x_t) \cap V(C_i)|\right)}\\
\ge& 2(2k-1)+(2k-1)+2(2k-1)+(2k-1)-16=12k-22>12(k-2),
\end{align*}
and the conclusion follows by averaging.
\end{claimproof}
We now fix an index~$i \in \{1,\dotsc,k-2\}$ be fixed such that the inequality from Claim~\ref{c9} holds.
For simplicity, we use the notation~$n_j \coloneqq |N^+(x_j)\cap V(C_i)|$ for~$j\in \{0,1,2,t\}$, so
\begin{equation}\label{eq9}
2n_0+n_1+2n_2+n_t\ge 13.
\end{equation}

\begin{claim}\label{c10}
  It holds that~$|V(C_i)|=3$.
\end{claim}

\begin{claimproof}
  Since none of~$x_1,x_2$ and~$x_t$ belongs to~$\bigcup_{j=1}^{k-1}V(C_j)$,
  Claim~\ref{c4} yields the conclusion if one of~$n_1,n_2,n_t$ is at least~$3$.
    Suppose therefore that~$n_1,n_2,n_t \le 2$. Together with~\eqref{eq9}, this implies that~$n_0\geq 3$.
    As noted earlier,~$D[\{u_t,u,u'\}]$ is a near-directed cycle, which is vertex-disjoint from
    each of~$C_1,\dotsc,C_{k-2}$. Consequently,~$\mathcal{C}'=\{C_1,\dotsc,C_{k-2},C\}$ is a collection of
    vertex-disjoint near-directed cycles in~$D$ covering as many vertices as~$\mathcal{C}$
    (since~$|V(C_{k-1})|=3$), and thus Claim~\ref{c4} applies to~$\mathcal{C'}$ and~$x_0$, yielding that~$|V(C_i)|=3$.
  \end{claimproof}

Claim~\ref{c10} implies that~$n_j \le 3$ for~$j\in \{0,1,2,t\}$.
Together with~\eqref{eq9}, this implies that~$2n_0+2n_2 \ge 7$, and so~$n_0+n_2 \ge 4$. 
For convenience, we set~$U \coloneqq V(C_i)\cup V(C_{k-1})\cup V(P)$.

\begin{claim}\label{c11}
  The digraph $D[U]$ contains no three vertex-disjoint near-directed cycles. As a consequence, none of the following five scenarios can occur.
\end{claim}
\begin{enumerate}[label=(\roman*)]
  \item\label{sci} $x_0, x_1$ have a common out-neighbor in~$V(C_i)$, while~$x_2$ sees the other two vertices of~$V(C_i)$.
  \item\label{scii} $x_2,x_t$ have a common out-neighbor in~$V(C_i)$, while~$x_0$ sees the other two vertices of~$V(C_i)$.
  \item\label{sciii} $x_1,x_2$ have a common out-neighbor in~$V(C_i)$, while~$x_0$ sees the other two vertices of~$V(C_i)$.
  \item\label{sciv} $x_2,x_t$ have a common out-neighbor in~$V(C_i)$, while~$x_1$ sees the other two vertices of~$V(C_i)$.
  \item\label{scv} $x_1,x_2$ have a common out-neighbor in~$V(C_i)$, while~$x_t$ sees the other two vertices of~$V(C_i)$.
\end{enumerate}

\begin{claimproof}
  If~$D[U]$ contained three disjoint near-directed cycles, then adding them to~$\mathcal{C}\setminus\{C_i,C_{k-1}\}$
  would yield~$k$ vertex-disjoint near-directed cycles in~$D$, contradicting our initial assumptions.

    For the second part, we show that assuming any of the five scenarios, we can construct three
    vertex-disjoint near-directed cycles in~$D[U]$, obtaining the desired contradiction. In each of the listed
    cases, let~$a$ be a common out-neighbor of the two vertices listed first and let~$b,b'$ be the other
    two vertices of the triangle~$C_i$, so~$V(C_i) = \{a,b,b'\}$. We now go through each of the five
    scenarios and identify the three disjoint near directed cycles in each case. In each of these cases, it is
    straightforward to verify that the given cycles are indeed vertex-disjoint and near-directed.

    \begin{itemize}
      \item[\ref{sci}] The triangles induced by~$\{x_0,x_1,a\}$, $\{x_t,u,u'\}$ and~$\{x_2,b,b'\}$ are three vertex-disjoint near-directed cycles in~$D[U]$.
        \item[\ref{scii}] Observation~\ref{obs}-\ref{obs-pathplusvertex} implies that~$D[\{a,x_2,\dotsc,x_t\}]$ contains a
          near-directed cycle which, together with the triangles induced by~$\{x_1,u,u'\}$ and~$\{x_0, b,b'\}$,
          yield three vertex-disjoint near-directed cycles in~$D[U]$.
        \item[\ref{sciii}] The triangles induced by~$\{x_1,x_2,a\}$, $\{x_t,u,u'\}$ and~$\{x_0,b,b'\}$ are three vertex-disjoint near-directed cycles in~$D[U]$.
        \item[\ref{sciv}] Observation~\ref{obs}-\ref{obs-pathplusvertex} implies that~$D[\{a,x_2,\dotsc,x_t\}]$ contains a
          near-directed cycle which, together with the triangles induced by~$\{x_0,u,u'\}$ and~$\{x_1, b,b'\}$,
          yield three vertex-disjoint near-directed cycles in~$D[U]$.
        \item[\ref{scv}] The triangles induced by~$\{x_1,x_2,a\}$, $\{x_0,u,u'\}$ and~$\{x_t,b,b'\}$ are three vertex-disjoint near-directed cycles in~$D[U]$.
\end{itemize}
\end{claimproof}

The forbidden conditions of Claim~\ref{c11} can be used to derive tighter bounds on~$n_0$ and~$n_2$,
eventually leading our minimal counterexample towards a contradiction.

\begin{claim}\label{c12}
  It holds that $n_0\le 2$.
\end{claim}
\begin{claimproof}
  Suppose, on the contrary, that $n_0=3$. Substituting this into~\eqref{eq9} yields that
    \[
    (n_1+n_2)+(n_2+n_t)\ge 7,
    \]
    which implies that~$n_1+n_2 \ge 4$ or~$n_2+n_t\ge 4$.
    Since two subsets of a three-element set whose sizes sum to at least four intersect, the first case implies that~$x_1$ and~$x_2$ have a common out-neighbor in~$V(C_i)$, while the second case implies that~$x_2$ and~$x_t$
    have a common out-neighbor in~$V(C_i)$. Since~$V(C_i)\subseteq N_D^+(x_0)$ as~$n_0=3$, we conclude that
    scenario~\ref{sciii} or~\ref{scii} of Claim~\ref{c11} holds, respectively, a contradiction.
\end{claimproof}

\begin{claim}\label{c13}
  It holds that $n_2 \le 2$.
\end{claim}
\begin{claimproof}
  Suppose, on the contrary, that $n_2 = 3$. Recalling that~$n_0 + n_2 \ge 4$, Claim~\ref{c12} implies that~$n_0 \in \{1,2\}$.
  We treat these two cases separately.

  If~$n_0=1$, then~\eqref{eq9} implies that~$n_1 + n_t \ge 5$, hence~$n_1\ge 2$.
    \begin{itemize}
        \item If~$n_1 = 2$, then~$n_t \ge 3$, so~$V(C_i)\subseteq N^+(x_2)\cap N^+(x_t)$. This means
          that~$x_2$ and~$x_t$ have a common out-neighbor in~$C_i$ such that~$x_1$ sees the other two vertices
          of~$C_i$, leading to scenario~\ref{sciv} of Claim~\ref{c11}, the desired contradiction.

        \item If~$n_1=3$, then~$x_0$ and~$x_1$ must have a common out-neighbor in~$C_i$, and~$x_2$ sees the other two vertices of~$C_i$, leading to scenario~\ref{sci}
          of Claim~\ref{c11}, the desired contradiction.
    \end{itemize}

    If~$n_0 = 2$, then~\eqref{eq9} implies that~$n_1+n_t \ge 3$ so~$n_1\ge 2$ or~$n_t \ge 2$.
    If~$n_t = 3$,
    since~$n_2 = 3$ and~$n_0 = 2$, we deduce that~$x_2$ and~$x_t$ must share an out-neighbor in~$V(C_i)$ such that~$x_0$ sees the
    other two vertices in~$V(C_i)$, leading to scenario~\ref{scii} of Claim~\ref{c11}, a contradiction. So~$n_t \le 2$.
    \begin{itemize}
        \item  If~$n_1 \ge 2$, then~$x_1$ shares an out-neighbor with~$x_0$ in~$V(C_i)$, and~$x_2$ sees the other two vertices of~$C_i$ as~$n_2=3$,
          leading to scenario~\ref{sci} of Claim~\ref{c11}, a contradiction. 
        
        \item If~$n_1 \le 1$, then since~$n_t\le 2$ we deduce that~$n_t=2$ and~$n_1=1$.
          Thus~$(n_0,n_1,n_2,n_t) = (2,1,3,2)$. Let~$y,y'$ denote the two distinct out-neighbors of~$x_t$
          in~$V(C_i)$. We assert that~$N^+(x_0)\cap V(C_i)=\{y,y'\}$. Indeed, if not, then let~$a$ be the
          unique element of~$\{y,y'\}$ not in~$N^+(x_0)$, and note that~$a$ is a common out-neighbor of~$x_2$ and~$x_t$.
          As~$x_0$ sees the other two vertices of~$C_i$, scenario~\ref{scii} of
          Claim~\ref{c11} occurs, a contradiction.

        Thus~$N^+(x_0)\cap V(C_i)=\{y,y'\}$. Now, if the unique out-neighbor~$x'$ of~$x_1$ in~$V(C_i)$
        lies in~$\{y,y'\}$, then~$x'$ is a out-neighbor of both~$x_0$ and~$x_1$. Since~$x_2$ sees the
        other two vertices of~$C_i$, scenario~\ref{sci} of Claim~\ref{c11} occurs, a contradiction.
        Otherwise,~$x'$ is an out-neighbor of both~$x_1$ and~$x_2$ contained in~$V(C_i)\setminus\{y,y'\}$, and both~$x_0$
        and~$x_t$ see the other two vertices~$y,y'$ of~$C_i$.
        Thus, both scenarios~\ref{sciii} and~\ref{scv} of Claim~\ref{c11} occur, yielding the desired
        contradiction also in this case. 
    \end{itemize}
\end{claimproof}

Claims~\ref{c12} and~\ref{c13} and the earlier observation that~$n_0 +n_2 \ge 4$ together imply that
values~$n_0=n_2=2$. Inequality~\eqref{eq9} then implies that~$n_1+n_t \ge 5$, therefore each of~$n_1$ and~$n_t$ is at least~$2$,
and one of them (at least) is~$3$.

\begin{itemize}
    \item If~$n_1 = 3$, then~$n_t\ge 2$ and~$n_2=2$ imply that~$x_2$ and~$x_t$ share an out-neighbor
      in~$V(C_i)$, and~$x_1$ sees the other two vertices of~$C_i$.
      This is scenario~\ref{sciv} of Claim~\ref{c11}, a contradiction.

    \item If~$n_t  =3$, then~$n_1 \ge 2$ and~$n_2=2$ imply that $x_1$ and~$x_2$ share an out-neighbor in~$V(C_i)$,
      and~$x_t$ sees the other two vertices of~$C_i$. This is scenario~\ref{scv} of Claim~\ref{c11}, a contradiction.
\end{itemize}
We have shown that every out-degree configuration in~$D[U]$ leads to a forbidden scenario listed in
Claim~\ref{c11}. This contradiction concludes the proof of Theorem~\ref{thm:main1}.
\end{proof}

\paragraph*{AI Disclosure.} No AI tools have been used for the writing or the generation of any of the mathematical ideas of this paper.

\bibliographystyle{alpha}
\bibliography{references}

\end{document}